\documentclass[12pt, a4paper]{amsart}

\usepackage{graphicx}
\usepackage{wrapfig}
\usepackage{comment}

\usepackage[
unicode=true,
plainpages = false,
pdfpagelabels,
bookmarks=true,
bookmarksnumbered=true,
bookmarksopen=true,
breaklinks=true,
backref=false,
colorlinks=true,
linkcolor = DarkBlue,
urlcolor  = DarkBlue,
citecolor = DarkRed,
anchorcolor = green,
hyperindex = true,
hyperfigures
]{hyperref}
\hypersetup{
pdftitle={A modified proof for Higman's embedding theorem},
pdfauthor={V.H. Mikaelian},
pdfsubject={A modified proof for Higman's embedding theorem}
}

\usepackage[dvipsnames]{xcolor}
\colorlet{LightRubineRed}{RubineRed!70!}
\colorlet{Mycolor1}{green!10!orange!90!}
\definecolor{DarkRed}{HTML}{cc0000}
\definecolor{ChapterHeadColor}{HTML}{cc0000}
\definecolor{PartHeadColor}{HTML}{cc0000}
\definecolor{DarkBlue}{HTML}{0000cc}
\definecolor{QuoteColor}{HTML}{665665}

\usepackage[top=2.5cm, bottom=2.5cm, outer=2.5cm, inner=2.5cm, heightrounded
]{geometry}

\newcommand{\B}{{\mathcal B}}
\renewcommand{\S}{{\mathcal S}}
\newcommand{\Zz}{{\mathcal Z}}
\newcommand{\E}{{\mathcal E}}

\newcommand{\X}{{\mathfrak X}}
\newcommand{\Y}{{\mathfrak Y}}

\newcommand{\BB}{\boldsymbol{\mathscr{S}}}

\newcommand{\Z}{{\mathbb Z}}
\newcommand{\Q}{{\mathbb Q}}

\newcommand{\1}{\{1\}}

\newcommand{\timesd}{\!\times\!}
\newcommand{\gic}{\bigskip {\color{green}\hrule} \bigskip}

\usepackage[bitstream-charter]{mathdesign}

\DeclareSymbolFont{cmsymbols}{OMS}{cmsy}{m}{n}
\SetSymbolFont{cmsymbols}{bold}{OMS}{cmsy}{b}{n}
\DeclareSymbolFontAlphabet{\mathcal}{cmsymbols}

\theoremstyle{plain}
\newtheorem{Theorem}{Theorem}[section]
\newtheorem{Lemma}[Theorem]{Lemma}
  
\newtheorem{Corollary}[Theorem]{Corollary}  
  
\theoremstyle{definition}
\newtheorem{Definition}[Theorem]{Definition}  
\theoremstyle{remark}
\newtheorem{Remark}[Theorem]{Remark} 
\newtheorem{Example}[Theorem]{Example} 

\numberwithin{equation}{section}

\usepackage[bitstream-charter]{mathdesign}

\DeclareSymbolFont{cmsymbols}{OMS}{cmsy}{m}{n}
\SetSymbolFont{cmsymbols}{bold}{OMS}{cmsy}{b}{n}
\DeclareSymbolFontAlphabet{\mathcal}{cmsymbols}

\begin{document}


\subjclass{20F05, 20E06, 20E07.}
\keywords{Recursive group, finitely presented group, embedding of group, benign subgroup, free product of groups with amalgamated subgroup, HNN-extension of group}

\title[A modified proof for Higman's embedding theorem]{\large  A modified proof for Higman's embedding theorem}

\author{V. H. Mikaelian
}

\begin{abstract}
We suggest a briefer  version for the proof of Higman's embedding theorem stating that a finitely generated group can be embedded into a finitely presented group if and only if it is recursively presented.
In particular, we shorten the main part of original proof establishing characterization of recursive relations in terms of  benign subgroups in free groups.
Also, the construction we suggest is adaptable for constructive  embeddings of recursive groups into explicitly given finitely presented groups.     
%
\end{abstract}

\date{\today}

\maketitle

\setcounter{tocdepth}{1}

\let\oldtocsection=\tocsection
\let\oldtocsubsection=\tocsubsection
\let\oldtocsubsubsection=\tocsubsubsection
\renewcommand{\tocsection}[2]{\hspace{-12pt}\oldtocsection{#1}{#2}}
\renewcommand{\tocsubsection}[2]{\footnotesize \hspace{6pt} \oldtocsubsection{#1}{#2}}
\renewcommand{\tocsubsubsection}[2]{ \hspace{42pt}\oldtocsubsubsection{#1}{#2}}

{\footnotesize \tableofcontents}

\section{Introduction}

\noindent
This note is a part of our research on Higman embeddings, also including recent articles \cite{Embeddings using universal words} and \cite{The Higman operations and  embeddings}.
Our first purpose here is to suggest a shorter and we hope considerably simpler version of the proof for Higman's remarkable embedding theorem:

\begin{Theorem}[Theorem 1 in~\cite{Higman Subgroups of fP groups}]
\label{TH Higman 1}
A finitely generated group can be embedded in a finitely presented group if and only if it is recursively presented.
\end{Theorem}

Recursive presentation for a group $G$ presumes that there is a presentation $G=\langle x_1, \ldots, x_n \mathrel{|} r_1, r_2, \ldots\, \rangle$ such that the set of defining relations $\{r_1, r_2, \ldots\}$ is recursive enumerable, i.e., it is the range of some recursive function.
This result establishes deep connections between the logical notion of recursion and theory of group.
For its applications we refer to 
Ch.~12 of \cite{Rotman},
Ch.~IV of \cite{Lyndon Schupp},
Ch.~I, IV, VI in \cite{Baumslag Topics in Combinatorial}, and to the work of Ol'shanskii and Sapir \cite{Ol'shanski Sapir The conjugacy problem}.

\medskip
And the second purpose of the   construction we suggest is that it is adaptable for \textit{explicit}  embeddings of recursive groups into finitely presented groups. In particular, using this embedding we in 
\cite{On explicit embeddings of Q} construct explicit embeddings of the group $\Q$ into finitely presented groups. 
The question of the existence of such group embeddings is posed in Problem 14.10\;(a) in Kourovka Notebook in 1999 \cite{kourovka}, see points 1.1, 1.2 in \cite{On explicit embeddings of Q} for details.

\subsection{The main steps of Higman's proof}
\label{SU The main steps of Higman's proof}
\cite{Higman Subgroups of fP groups} starts by Kleene's formal characterization of partial recursive functions on the set of non-negative integers 
as the class of
functions that can be obtained from the {zero}, 
 {successor} 
and  
{identity} functions
using the operations of  composition, primitive recursion, and minimization (see \cite{Davis, Rogers} or the newer text \cite{Boolos Burgess Jeffrey}).
Let $\Z$ be the set of all integers, and $\mathcal E$ be the set of all functions $f : \Z \to \Z$ with finite support in the sense that $f(i)=0$ for all but finitely many integers $i\in \Z$. To each such function $f\in \E$ one can constructively assign a unique non-negative  G\"odel number.
This makes $\mathcal E$ an effectively enumerable set, and one may put into correspondence to each subset $\B$ of $\E$ the respective set of non-negative integers corresponding to the functions $f\in \B$.
Then $\B$ is \textit{recursively enumerable}, if that set is the \textit{range} of some recursive function. 
Next, Section 2 in \cite{Higman Subgroups of fP groups} gives a different characterization  for recursively enumerable subsets of $\E$: two initial subsets $\Zz$ and $\S$ of $\E$, and a series of operations on subsets of $\E$ are introduced (see \eqref{EQ Higman operations} in \ref{SU The Higman operations on subsets} below), and  Theorem 3 states that $\B$ is recursively enumerable if and only if it can be obtained from $\Zz$ and $\S$ by operations \eqref{EQ Higman operations}.
Only after these preparations the  group-theoretical argument starts in~\cite{Higman Subgroups of fP groups}. The key concept of \textit{benign subgroup} is introduced (see \ref{SU Definition and main properties of benign subgroups}), and to each $\B$ a specific subgroup $A_\B$ of the free group $F=\langle a,b,c \rangle$ of rank $3$ is put into correspondence. The main result of Section 3 and of Section 4 in \cite{Higman Subgroups of fP groups} is:

\begin{Theorem}[Theorem 4 in~\cite{Higman Subgroups of fP groups}]
\label{TH Higman 4}
The subset $\mathcal B$ of $\mathcal E$ is recursively enumerable if and only if $A_{\mathcal B}$ is a benign subgroup in $F=\langle a,b,c \rangle$.
\end{Theorem}

Then in final brief Section 5 the above theorem is generalized:
a subgroup of \textit{any} finitely generated free group is recursively enumerable if and only if it is benign. The ``Higman Rope Trick'', as it is sometimes 
called~\cite{Higman rope trick}, closes the proof by showing how a recursively presented group can be embedded in a finitely presented group.
Here we suggest a shorter proof for Theorem~\ref{TH Higman 4}, and our sections \ref{SE The main properties and basic examples of benign subgorups}, \ref{SE The proof of Theorem} roughly correspond to sections 3, 4 in \cite{Higman Subgroups of fP groups}.


\subsection{Comparison of the current modification with~\cite{Higman Subgroups of fP groups}}
Leaving aside minor changes, here are the main modifications done to Higman's original construction.

We construct a very different group to show that
the set of all subsets $\mathcal B \subseteq \E$, for which the subgroup $A_{\mathcal B}= \langle a_f \;|\; f \in \mathcal B\, \rangle$ is benign, is closed under Higman operation $\omega_m$. In \cite{Higman Subgroups of fP groups} this is done by the group $M *_{\langle a,b,c \rangle} K$ built using Lemma 4.10, Lemma 3.9 and Lemma 3.10, which perhaps form the most complicated part in \cite{Higman Subgroups of fP groups}. Compare them with the much shorter part ``\textit{$\BB$ is closed under  $\omega_m$}\!''  in Section~\ref{SE The proof of Theorem} below.

Three lemmas 3.2--3.4 on homomorphisms are extensively used in~\cite{Higman Subgroups of fP groups} to study subgroups in free constructions. We obtain similar results using really simple combinatorics on words, and  trivial observations made in \ref{SU The conjugates collecting process}. Compare, for instance, 
Lemma 3.8 and Lemma 3.9 of ~\cite{Higman Subgroups of fP groups} with direct short proofs in examples listed in \ref{SU Basic examples of benign subgroups} below.

We argument some passages that are just stated in \cite{Higman Subgroups of fP groups}. In particular, 
we do prove the passage 
\textit{``These two conditions are easy, but a little tedious, to check, and this is left to the reader''} on p.\,472 in  \cite{Higman Subgroups of fP groups},  see the part ``\textit{$\BB$ is closed under  $\omega_m$}\!'' below.
Also, we accompany all concepts by examples which make understanding less tedious, we hope. 
%
In fact, we could shorten the proofs even more by involving some wreath product methods from \cite{Subnormal embedding theorems}--\cite{Subvariety structures}. However, we intentionally keep  technique within free products with amalgamations and  HNN-extensions to preserve Higman's original idea about investigation of recursion via free constructions.


\subsection{Other proofs for Higman's embedding theorem}

Higman's original proof is very complicated, and this motivated others to suggest alternative proofs for Theorem~\ref{TH Higman 1}. 
Lindon and Schupp present in Section IV\,\!.\,\!7 of \cite{Lyndon Schupp} a proof related to Valiev's approach \cite{Valiev}.
Solving Hilbert's Tenth Problem \cite{Matiyasevich Diophantine} establishes  that a subset of $\Z^n$ is recursively enumerable if and only if it is Diophantine.
The proof in \cite{Lyndon Schupp}  uses this Diophantine characterization, i.e., it relies on a ``third party'' result  (which itself demands a complicated proof). 

A proof reflecting  Aanderaa's  work \cite{Aanderaa} is given by Rotman in Ch.\,12 of \cite{Rotman}. It applies the auxiliary construction  interpreting Turing machines via semigroups, and also uses the Boone-Britton group  \cite{Boone Certain simple unsolvable problems, Britton The word problem}. Group diagrams allow \cite{Rotman} to shorten the proof of \cite{Aanderaa}.

Besides the mentioned two well known textbooks \cite{Lyndon Schupp, Rotman}, other proofs can be found in the Appendix of Shoenfield's textbook on logic  \cite{Shoenfield}, in the article of Adyan and Durnev \cite{Adyan Durnev}, etc. Typically, they relay on auxiliary  constructions, such as those  built to show undecidability of the word problem in semigroups \cite{Post Recursive unsolvability, 
Markov impossibility of certain algorithms} and in  groups
\cite{Novikov on algorithmic undecidability, 
Boone Certain simple unsolvable problems, 
Britton The word problem}.

\smallskip
This comparison stresses importance of Higman's straightforward approach of \textit{mimicking recursion} by means of some elegant group-theoretical constructions only,  without resorting to ``third party'' concepts.
A motivation of our note is to emphasize that Higman's original idea still is one of the clearest approaches to handle this subject.

\medskip
We would like to also announce the recent article \cite{The Higman operations and  embeddings} in which we use the current embedding as a basis to suggest an algorithms, called $H$-Machine, which explicitly builds the subset $\B$ of $\E$
for certain classes of groups,
such as the free abelian, metabelian, soluble, nilpotent groups, the groups
$\Q$ and $\mathbb C_{p^\infty}$, divisible abelian groups, etc.


\subsection*{Acknowledgements}
The current work is supported by the 
21T-1A213 grant of SCS MES RA.
I would like to thank  Yerevan State University  for the Teaching Excellence award and grant for 2021.
The first steps of this work was announced at the International Algebraic Conference in memory of A.G.~Kurosh, May 23--25, 2018.


\section{Notations and preliminary information}

\subsection{Sets of integer-valued funations}
\label{SU Sets of integer-valued funations}

In~\ref{SU The main steps of Higman's proof} we introduced the set  $\mathcal E$ of all functions $f : \Z \to \Z$ with finite supports.
If $f(i)=0$ for any $i<0$ and $i\ge n$, we may record such a function as $f=(a_0,\ldots,a_{n-1})$ assuming $f(i)=a_i$ for $i=0,\ldots,n\!-\!1$. Say,
$f=(0,0,7,-8,5,5,5,5)$
means 
$f(2)=7$, $f(3)=-8$, 
$f(i)=5$ for $i=4,\ldots,7$, and 
$f(i)=0$ for any $i<2$ or $i\ge 8$.
Denote by $f=(0)$ the constant zero function on $\Z$.
For a fixed integer $m$ and for an $f\in \mathcal E$  denote the functions $f_{m}^+$ and $f_{m}^-$ as follows. 
$f_{m}^+(i)=f_{m}^-(i)=f(i)$ for all $i\!\neq\! m$, and   
$f_{m}^+(m) = f(m) +1$,\; 
$f_{m}^-(j) = f(m) -1$.
When $m$ is given by the context, we may shorten these to $f_{m}^+=f^+$ and $f_{m}^-=f^-$. Say, for the above $f$ we get $f_{5}^+\!\!=(0,0,7,-8,5,6,5,5)$, and 
$f^-\!\!=(0,0,7,-8,5,5,5,4)$, provided $m=7$ is given.


\subsection{The Higman operations on subsets of $\E$}
\label{SU The Higman operations on subsets}

Define some specific subsets of $\E$:
$$
\Zz\!=\!\big\{(0)\big\},\;\;  
\S\!=\!\big\{(a, a+\!1) \mathrel{|} a\!\in\! \Z\big\},\;\; 
\E_m\!\!=\!\big\{
(a_0,\ldots,a_{m-1})
\mathrel{|} a_i\!\in \!\Z,\; i\!=\!0,\ldots,\!m-1\big\},
\, m\!=\!0,\!1,\ldots
$$
($\E_0$ consists of zero function only).
We will use the following  operations from \cite{Higman Subgroups of fP groups}:
\begin{equation}
\label{EQ Higman operations}
\tag{H}
\iota,\; \upsilon;\quad\quad
\rho,\; \sigma,\; \tau,\; \theta,\; \zeta,\; \pi,\; \omega_m \;\; (\text{for each $m=1,2,\ldots$})
\end{equation}
which we call \textit{Higman operations}. 
For any subsets $\mathcal A, \mathcal B$ of $\E$ define 
$\iota(\mathcal A,\mathcal B)=\mathcal A \cap \mathcal B$ and 
$\upsilon(\mathcal A,\mathcal B)=\mathcal A \cup \mathcal B$.
The rest of Higman operations  are unary functions on the subsets of $\mathcal E$:

$f \in \rho(\mathcal A)$ iff there is a $g\in \mathcal A$ such that $f(i)=g(-i)$,

$f \!\in\! \sigma(\mathcal A)$ iff there\! is\! a\! $g\in \mathcal A$ such\! that\! $f(i)=g(i-1)$,

$f\! \in \tau(\mathcal A)$ iff there is a $g\!\in\! \mathcal A$ such that $f(0)\!=\!g(1)$, $f(1)\!=\!g(0)$ and  $f(i)\!=\!g(i)$ for $i\!\neq\! 0,1$,

$f \in \theta(\mathcal A)$ iff there is a $g\in \mathcal A$ such that $f(i)=g(2i)$,

$f \in \zeta(\mathcal A)$ iff there is a $g\in \mathcal A$ such that $f(i)=g(i)$ for  $i\neq 0$,

$f \in \pi(\mathcal A)$ iff there is a $g\in \mathcal A$ such that $f(i)=g(i)$ for  $i\le 0$,

$f \in \omega_m(\mathcal A)$ iff for every $i\in \Z$ there is a  $g = \big(f(mi),\;f(mi+1),\ldots,f(mi+m-1)\big)\in \mathcal A$. Since  the support of any $f\in \E$ is finite, either $\mathcal A$ contains the zero function, or $\omega_m(\mathcal A)=\emptyset$.

\smallskip
The more general operations $\alpha$ and $\lambda$ mentioned \cite{Higman Subgroups of fP groups} are not used here.

\begin{Example} 
\label{EX simple Higman operations}
Some simple applications of Higman operations are easy to check:

$\rho(\S)
=\big\{f\in \E \mathrel{|}
f(-1)=f(0)+1,  f(i)=0 \text{ for } i\neq -1,0
\big\}
$.

Thus, $\sigma\rho(\S)=\big\{(a+\!1,a) \mathrel{|} a\!\in\! \Z\big\}=\tau(\S)$.

Since $\zeta(\Zz)=
\big\{(a) \mathrel{|} a\in \Z\big\}=
\big\{(a,0) \mathrel{|} n\in \Z\big\}$, we have $\iota\big(\!\tau(\S),\zeta(\Zz)\big)=\big\{(1,0)\big\}=\big\{(1)\big\}$.

As $\sigma \zeta (\Zz)= \big\{(0,a) \mathrel{|} a\in \Z\big\}$, we have
$\zeta (\sigma \zeta)^{2} (\Zz)=
\big\{ (a_0,a_1,a_2) \mathrel{|} a_0,a_1,a_2\in \Z\big\}
= \E_3$.

Since $\pi (\Zz)=\{f\in \E \mathrel{|}\
f(i)=0 \text{ for } i\le 0
\}$, then
$\sigma^3 \rho \pi (\Zz)=\{f\in \E \mathrel{|}\
f(i)=0 \text{ for } i\ge 3
\}$.

Clearly, $\theta (\E_4)=
\big\{ (a_0,a_2) \mathrel{|} \text{ for each }  (a_0,a_1,a_2,a_3)\in \E_4\big\}
= \E_2$.

$\omega_2\big(\!\upsilon(\S, \Zz)\big)=\{f\in \E \mathrel{|}\
f(2i+1)=f(2i)+1 \text{ \,or\,} f(2i+1)=f(2i)=0 \text{ for any } i\in \Z
\}$. Say, $(7,8,\;0,0,\; 2,3)\in \omega_2\big(\!\upsilon(\S, \Zz)\big)$ because $(7,8), (2,3)\in \S$ and $(0,0)\in \Zz$.
Also, $\omega_2(\S)=\emptyset$.

\end{Example}

\subsection{Notations for groups and homomorphisms}
\label{SU Notations for groups and homomorphisms}

Below we may use the functions  $f\in \mathcal E$ to define specific elements in free groups. 
Here is a typical context: if $F=\langle a,b,c \rangle$ is a free group of rank $3$, then the elements $b_i = b^{c^i} = c^{-i} b\, c^i$,\, $i\in \Z$, are generating a free subgroup  of countable rank in $F$. 
For every  $f\in \E$ we can define the element
$b_f= \cdots b_{-1}^{f(-1)} b_{0}^{f(0)} b_{1}^{f(1)} \cdots$ 
In particular, when $f=(a_0,\ldots,a_{m-1})\in \mathcal E$, we  define $b_f= b_0^{f(0)}\cdots b_{m-1}^{f(m-1)}$\!.\,
Say, for $f=(2,5,3)$ we have 
$b_f= b_0^{2}\, b_1^{5} \, b_2^{3}
= (b^{c^0})^2
(b^{c^1})^5
(b^{c^2})^3
=b^2 c^{-1} b^5 c^{-1} b^3 c^2
$.
Further, the elements of type $a_f=a^{b_f}$ generate yet another free subgroup in $F$.

\smallskip
If $G = \langle\, X \mathrel{|} R \,\rangle$ is the presentation of the group $G$ by its generators $X$ and definning relations $R$, then for an alphabet $Y$ (disjoint from $X$) and for any set $S\subseteq F_Y$ of words on $Y$ we denote
$\langle G, Y \mathrel{|} S \,\rangle = \langle \,X \cup Y \mathrel{|}  R \cup S \,\rangle$ 
(the cases when $X = \emptyset$ or $S = \emptyset$ are not ruled out).

If $\varphi:G\to H$ is a homomorphism defined on any group $G=\langle a,b,\ldots \rangle$ by the images $\varphi(a)=a'$,  $\varphi(b)=b',\ldots$\; of its generators $a,b,\ldots$\,, we may for briefness refer to $\varphi$ as \textit{the homomorphism sending} $a,b,\ldots$ \;to\; $a',b',\ldots$

\subsection{Notations for free constructions}
\label{SU Notations for free constructions}

We are going to extensively use the operations of free products of groups with amalgamated subgroups and of HNN-extensions of groups by one or more stable letters.
Referring for background information to \cite{Lyndon Schupp, Bogopolski,Rotman}, here we just state the notations we use.

If the groups $G$ and $H$ have subgroups, respectively, $A$ and $B$ isomorphic under the isomorphism $\varphi : A \to B$, then we denote the (generalized) free product $\Gamma = \langle G, H \mathrel{|} a=a^{\varphi} \text{ for all $a\in A$}\, \rangle$ of $G$ and $H$ with amalgamated subgroups $A$ and $B$ by
$G*_{\varphi} \!H$. In the simplest case, when $G$ and $H$ are overgroups of the same subgroup $A$, and $\varphi$ is just the \textit{identical} isomorphism on $A$, we may prefer to write $\Gamma = G*_{A} H$.

If the group $G$ has subgroups $A$ and $B$ isomorphic under the isomorphism $\varphi : A \to B$, then we denote the HNN-extension $\Gamma = \langle G, t \mathrel{|} a^t=a^{\varphi} \text{ for all $a\in A$}\, \rangle$ of the base $G$ 
by the stable letter $t$
with respect to the isomorphism 
$\varphi$ by
$G*_{\varphi} t$.
In case when $A=B$ and $\varphi$ is  \textit{identity} on $A$, we may write $\Gamma = G*_{A} t$.
We also use HNN-extensions with more than one stable letters. If we have the isomorphisms $\varphi_1: A_1 \to B_1,\; \varphi_2: A_2 \to B_2,\ldots$ for pairs of subgroups in $G$,  we denote the respective HNN-extension $\langle G, t_1, t_2,\ldots \mathrel{|} a_1^{t_1}=a_1^{\varphi_1}\!\!,\; a_2^{t_2}=a_2^{\varphi_2}\!\!,\,\ldots\; \text{ for all $a_1\in A_1$, $a_2\in A_2,\ldots$}\, \rangle$ by $G *_{\varphi_1, \varphi_2, \ldots} (t_1, t_2, \ldots)$. 
%
Our usage of the \textit{normal forms} in
free constructions is closer to~\cite{Bogopolski}.

\subsection{Subgroups in free constructions}
\label{SU Subgroups in free constructions}

Lemma~\ref{LE subgroups in amalgamated product} 
is a slight variation of Lemma 3.1 given in \cite{Higman Subgroups of fP groups} without a proof as ``obvious from the normal form theorem for free products with an amalgamation''.
Lemma~\ref{LE subgroups in HNN-extension}
is its analog for HNN-extensions.
See \cite{Auxiliary free constructions for explicit embeddings} for proof details.

\begin{Lemma}
\label{LE subgroups in amalgamated product}
Let $\Gamma = G *_\varphi \!H$ be the free product of the groups $G$ and $H$ with amalgamated subgroups
$A \le G$ and $B \le H$
with respect to the isomorphism 
$\varphi: A \to B$.
If $G', H'$ respectively are subgroups of $G, H$, such that
for $A'=G'\cap A$ and $B'=H'\cap B$ 
we have  
$\varphi (A') = B'$, then for the subgroup $\Gamma'=\langle G',H'\rangle$  of $\Gamma$ and for the restriction $\varphi'$ of $\varphi$ on $A'$ we have:
\begin{enumerate}
\item 
\label{PO 1 LE subgroups in amalgamated product}
$\Gamma' = G'*_{\varphi'} H'$,

\item 
\label{PO 2 LE subgroups in amalgamated product}
$\Gamma' \cap A = A'$ and $\Gamma' \cap B=B'$,

\item 
\label{PO 3 LE subgroups in amalgamated product}
$\Gamma' \cap G = G'$ and $\Gamma' \cap H = H'$.
\end{enumerate}
\end{Lemma}

\begin{Corollary} 
\label{CO NEW G*H lemma corollary}
Let $\Gamma = G*_{A} H$,
and let $G'\le G$, $H'\le H$ be subgroups for which $G' \cap\, A = H' \cap\, A$. Then for $\Gamma'=\langle G', H'\rangle$ and $A' = G' \cap\, A $ we have:

\begin{enumerate}

\item 
\label{PO 1 in CO NEW G*H lemma corollary}
$\Gamma'= G' \!*_{A'} H'$, in particular, if  $A\le G'\!,\,H'$, then $\Gamma'= G' *_{A} H'$;

\item 
\label{PO 2 in CO NEW G*H lemma corollary}
$\Gamma' \cap\, A = A'$, in particular, if  $A\le G'\!,\,H'$, then $\Gamma' \cap\, A = A$\,;

\item 
\label{PO 3 in CO NEW G*H lemma corollary}
$\Gamma' \cap\, G = G'$ and 
$\Gamma' \cap\, H = H'$.

\end{enumerate}
\end{Corollary}

\begin{Lemma}
\label{LE subgroups in HNN-extension}
Let $\Gamma=G *_{\varphi} t$ be the HNN-extension of the base group $G$ by the stable letter $t$
with respect to the isomorphism 
$\varphi: A \to B$
of the subgroups 
$A, B \le G$.
If $G'$ is a subgroup of $G$  such that
for $A'=G'\cap\, A$ and $B'=G'\cap\, B$ 
we have  
$\varphi (A') = B'$, then for the subgroup $\Gamma'=\langle G',t\rangle$  of $\Gamma$ and for the restriction $\varphi'$ of $\varphi$ on $A'$ we have:
\begin{enumerate}
\item 
\label{PO1 LE subgroups in HNN-extension}
$\Gamma' =G' *_{\varphi'} t$,

\item 
\label{PO2 LE subgroups in HNN-extension}
$\Gamma' \cap G = G'$,

\item 
\label{PO3 LE subgroups in HNN-extension}
$\Gamma' \cap \,A = A'$ \;and\; $\Gamma' \cap B = B'$.
\end{enumerate}
\end{Lemma}

\begin{Corollary} 
\label{CO NEW G*t lemma corollary}
Let $\Gamma = G*_{A} t$,
and let $G'\le G$ be a subgroup. Then for $\Gamma'=\langle G', t\rangle$ and $A' = G' \cap\, A $ we have:

\begin{enumerate}

\item 
\label{PO 1 in CO NEW G*t lemma corollary}
$\Gamma'= G' \!*_{A'} t$, in particular, if  $A \le G'$, then $\Gamma'= G' *_{A} t$;

\item 
\label{PO 2 in CO NEW G*t lemma corollary}
$\Gamma' \cap\, A = A'$, in particular, if  $A\le G'$, then $\Gamma' \cap\, A = A$\,;

\item 
\label{PO 3 in CO NEW G*t lemma corollary}
$\Gamma' \cap\, G = G'$.

\end{enumerate}
\end{Corollary}

\begin{Remark} 
\label{RE about multiple stable letters}
It is easy to adapt Lemma~\ref{LE subgroups in HNN-extension} for the case of multiple stable letters $t_1, \ldots, t_k$. And that adaptation will be especially simple, if all the stable letters $t_1, \ldots, t_k$ just fix the \textit{same} subgroup $A$ in $G$. In such a case, say, point~\eqref{PO 1 in CO NEW G*t lemma corollary} in Corollary~\ref{CO NEW G*t lemma corollary} will read: 
$\Gamma'= G' \!*_{A'} (t_1, \ldots, t_k)$ for $\Gamma'=\langle G', t_1, \ldots, t_k\rangle$ and for $A' = G' \cap\, A$. 
\end{Remark}

\subsection{The ``conjugates collecting'' process}
\label{SU The conjugates collecting process} 

Let $\X$ and $\Y$ 
be any disjoint subsets in a group $G$. It is easy to see that any element $w\in \langle \X,\Y \rangle$ can be written as:
\begin{equation}
\label{EQ elements <X,Y>}
w=u\cdot v
= x_1^{\pm v_1}
x_2^{\pm v_2}
\cdots
x_{k}^{\pm v_k}
\cdot
v
\end{equation}
with some $v_1,v_2,\ldots,v_k,\, v\in\langle \Y \rangle$, and 
$x_1,x_2,\ldots,x_k \in \X$.
Indeed, since $w$ is in $\langle \X,\Y \rangle$, present $w$ as a product of elements from $\X$, from $\Y$ and of their inverses. Then by grouping where necessary some elements from $\Y$, and by adding some trivial elements we can rewrite it as:
\begin{equation}
\label{EQ preliminary step}
w=z_1^{\vphantom8} x_1^{\pm1} 
z_2^{\vphantom8} \, x_2^{\pm1} z_3^{\vphantom8}
\cdots
z_k^{\vphantom8}\, x_k^{\pm1} z_{k+1}^{\vphantom8}
\end{equation}
where $x_1,\ldots,x_k \in \X$ and 
$z_1,\ldots,z_{k+1} \in \langle \Y \rangle$.
For instance, if $\X =\{x_1, x_2, x_3\}$ and  
$\Y =\{y_1, y_2\}$, then 
$w=x_2^{-1} y_1^3  y_2^{\vphantom8} x_1^2 x_3^{\vphantom8}$ can be rewritten as 
$w=z_1\, x_2^{-1} z_2^{\vphantom8} x_1 z_3^{\vphantom8} x_1
z_4^{\vphantom8}
x_3^{\vphantom8}
z_5^{\vphantom8}
$,
where
$z_1=1$,\,
$z_2=y_1^3  y_2^{\vphantom8}$,\,
$z_3=z_4=z_5=1$ are in $\langle \Y \rangle$. Next, \eqref{EQ preliminary step} can be transformed to:
$$w=z_1^{\vphantom8} x_1^{\pm1}
z_1^{-1} \!\cdot\,
z_1^{\vphantom8}
z_2^{\vphantom8} \, x_2^{\pm1}
(z_1^{\vphantom8}
z_2^{\vphantom8})^{-1}
\!\cdot\,
z_1^{\vphantom8}
z_2^{\vphantom8}
z_3^{\vphantom8}
\cdots
(z_1^{\vphantom8}\!\cdots z_k^{\vphantom8})\, x_k^{\pm1}
(z_1^{\vphantom8}\!\cdots z_k^{\vphantom8})^{-1}
z_1^{\vphantom8}\!\cdots z_k^{\vphantom8} z_{k+1}^{\vphantom8},
$$
which is \eqref{EQ elements <X,Y>} for
$v_1^{\vphantom8}\!=z_1^{-1}$\!\!\!, \;
$v_2^{\vphantom8}\!=(z_1z_2)^{-1}$\!\!\!\!, \; 
$v_3^{\vphantom8}\!=(z_1 z_2 z_3)^{-1}$\!\!\!\!, \, $\ldots$
$v_k^{\vphantom8}\!=(z_1 \cdots z_k)^{-1}$\!\!\!\!, \;  
$v\!=z_1^{\vphantom8}\!\cdots\; z_k^{\vphantom8} z_{k+1}^{\vphantom8}$.

In particular, setting $\X=\{x\}$ and $\Y=\{y\}$ in a 2-generator group $G=\langle x, y\rangle$ we can present any element $w \in G$ as a product
of some conjugates of $x$ and of some power of $y$:
\begin{equation}
\label{EQ elements from <x,y>}
w=x^{\pm y^{n_1}}\!x^{\pm y^{n_2}}\!\!\cdots\, x^{\pm y^{ n_s}}\!\!\cdot y^k=\!u\cdot v\,.
\end{equation}


\section{The main properties and basic examples of benign subgroups}
\label{SE The main properties and basic examples of benign subgorups}

\subsection{Definition and main properties of benign subgroups}
\label{SU Definition and main properties of benign subgroups}
The central group-theoretical notion used in \cite{Higman Subgroups of fP groups} to construct embeddings into finitely presented groups is the concept of benign subgroups. Higman gives three equivalent definitions of which we use the first one:

\begin{Definition}
\label{DE benign subgroup}
A subgroup $H$ of a finitely generated group $G$ is called a \textit{benign subgroup} in $G$, if $G$ can be embedded in a finitely presented group $K$ which has a finitely generated subgroup $L$ such that $G \cap L = H$.
\end{Definition} 

As a first evident example of a benign subgroup one may take any finitely generated subgroup $H$ in any finitely presented group $G$. Then we just have to choose $L\!=\!H$.
%
%
It is interesting to compare Definition~\ref{DE benign subgroup} with Theorem V (a variation of the well-known theorem on embeddings of countable groups into 2-generator groups) in earlier article \cite{HigmanNeumannNeumann}.

\medskip

The following properties have simple proofs covered in \cite{Higman Subgroups of fP groups} by Lemma~3.6 and Lemma~3.7:

\begin{Lemma}
\label{LE benign subgroups intersection and union are benign}
Let $G$ be a finitely generated group with two benign subgroups  $H,\,H'$. Then:
\begin{enumerate}
\item $H \cap H'$  is benign in $G$,
\label{PO 1 LE benign subgroups intersection and union are benign}

\item 
\label{PO 2 LE benign subgroups intersection and union are benign}
$\langle H, H' \rangle$ is benign in $G$.
\end{enumerate} 
\end{Lemma}

\begin{Lemma}
\label{LE benign subgroup has benign image}

Let $G,\, N$ be finitely generated groups both possessing embeddings into some finitely presented groups, and let $\varphi : G \to N$ be any homomorphism. Then:
\begin{enumerate}
\item 
\label{PO 1 LE benign subgroup has benign image}
if a subgroup $H$ is benign in $G$, then its image $\varphi(H)$ is benign in $N$,

\item 
\label{PO 2 LE benign subgroup has benign image}
if a subgroup $H$ is benign in $N$, then its (complete) preimage $\varphi^{-1}(H)$ is benign in $G$.
\end{enumerate} 
\end{Lemma}

\begin{Remark}
\label{LE benign in two subgroups}
It is trivial to verify by Definition~\ref{DE benign subgroup} that the subgroup $H$ of a finitely generated group $G$ is benign in $G$, provided that $H$ is benign in a finitely generated group $G'$ containing $G$. 
On the other hand, if we additionally require that $G'$ (and, therefore, also $G$) can be embedded into a finitely presented group, then 
$H$ is benign in $G'$, provided that $H$ is benign in $G$. To see this just apply
\eqref{PO 1 LE benign subgroup has benign image} in 
Lemma~\ref{LE benign subgroup has benign image} to the identical embedding $G \to G'$.
Combining these facts, let us following \cite{Higman Subgroups of fP groups} agree to call $H$ \textit{benign} (without specifying in \textit{which} group $G$), whenever there is a finitely presented overgroup containing the groups discussed.
\end{Remark}

\subsection{Examples of benign subgroups}
\label{SU Basic examples of benign subgroups}

Let us collect some examples for later use. We intentionally bring them in an order displaying consecutive development of some basic ideas.
Take the free group $F=\langle a,b,c \rangle$, its subgroup 
$F_0=\langle b,c \rangle$, and denote $b_i=b^{c^i}$ for $i \in \Z$ as in subsection \ref{SU Notations for groups and homomorphisms}.

\begin{Example} 
\label{EX b all powers is benign}
The subgroup $A_0= \langle b_i \mathrel{|} 
i\in \Z \,\rangle$ is benign in $F_0$.
Indeed, let $\varphi$ be the isomorphism of $F_0$ sending $b,c$ to $b^c\!,c$.
Then $\Gamma\!=\! F_0\, \!*_{\varphi}\! t$ is finitely presented, and it is enough to verify 
$F_0 \cap \langle b, t \rangle \!= A_0$.
Each $b_i\!=b^{c^i}\!\!=b^{t^i}$ is in $F_0 \cap \langle b, t \rangle$.
Next, applying \ref{SU The conjugates collecting process}
to any $w\in \langle b, t \rangle$
for $\X=\{b \}$
and $\Y=\{t\}$, we get $w = u \cdot v$ where $u$ is a product of factors of type $b^{\pm t^{i}}\!\!= b_i^{\pm 1}$\!, and $v=t^k$. If also $w\in F_0$, then $v=1$ by uniqueness of the normal form, i.e., $w\in A_0$.
\end{Example}

\begin{Example} 
\label{EX b SOME powers is benigh}
The subgroup $A_{k,l}=\langle b_{n} \mathrel{|} 
n= ik+l,\;i\in \Z\,
\rangle$ is benign in $F_0$ for any $k,l=1,2,\ldots$ To see this just apply to $A_{\,0}$ the homomorphism $\varphi$ sending $b,\,c$ to $b^{c^l}\!\!,\;c^k$. 
Since $A_{k,l}=\varphi(A_0)$, it remains to use 
\eqref{PO 1 LE benign subgroup has benign image} in 
Lemma~\ref{LE benign subgroup has benign image}.
\end{Example}


\begin{Example}
\label{EX B_0 is benign}
The subgroup $B_0=\langle b_i \mathrel{|} i\ge 0 \rangle$ is benign in $F_0$.
Start the proof by denoting $F'=\langle
b_i \;|\; \text{$i$\, is odd} \rangle$ and 
$F''=\langle
b_i \;|\; \text{$i$\, is even} \rangle$, and observing that $F'*F'' = A_0 = \langle b_i \;|\;  i \in \Z\rangle$.
Construct a transversal to $F'$ in $A_0$ using the fact that $A_0$ is \textit{freely} generated by all $b_i$. Write any word $w\in A_0$ as:
\begin{equation}
\label{EQ the worm of w OLD}
w=b_{n_1}^{\varepsilon_1}\!\cdots b_{n_{r-1}}^{\varepsilon_{r-1}}\,\cdot\,
b_{n_r}^{\varepsilon_r}\!\cdots b_{n_s}^{\varepsilon_s},
\end{equation}
and mark $r$ to be the first \textit{even} index, i.e., all  $n_1,\ldots,n_{r-1}$ are odd (the case $r=1$ is not ruled out). Then $b_{n_1}^{\varepsilon_1}\!\cdots b_{n_{r-1}}^{\varepsilon_{r-1}} \in F'$, and the words or type $b_{n_r}^{\varepsilon_r}\!\cdots b_{n_s}^{\varepsilon_s}$ with even starting index $n_r$ form the transversal $T_{F'}$ to $F'$ in $A_0$.

Next, show that some $u\in F_0$ is trivial modulo $F'$ if and only if it is trivial modulo $B=\langle F', c^2 \rangle$. Indeed, applying \ref{SU The conjugates collecting process} to $u\in B$ for $\X=\{b_i \;|\; \text{$i$\, is odd}\}$ and $\Y=\{c^2\}$, we can write it as
$u=
b_{i_1}^{\pm (c^2)^{n_1}}\!\!\cdots\, 
b_{i_s}^{\pm (c^2)^{n_s}}\! \cdot c^{2k}
\;=\;b^{\pm 1}_{i_1+2 n_1}\cdots\, b^{\pm 1}_{i_s+2 n_s}\cdot  c^{2k}
$ 
for some odd
$i_1,\ldots,i_s$
and for some
$n_1,\ldots,n_s,\, k \in \Z$.
Since evidently $F_0=A_0 *_\pi c$ (for the isomorphism  $\pi:b \to b^c$), then by uniqueness of the normal form $u\in A_0$ holds if $c^k = 1$.
But since each of the indices 
$i_1\!+\!2n_1,\;\ldots\;,i_s\!+\!2n_s$ is odd, then also
$u\in F'$.
Hence, the above found transversal $T_{F'}$ to $F'$  can be continued to a transversal  $T_{B}$ to $B$ in the whole $F_0$.

In a similar manner define $D=\langle F'', c^2 \rangle=\langle b, c^2 \rangle$, and show that the elements of type $b_{n_r}^{\varepsilon_r}\!\cdots b_{n_s}^{\varepsilon_s}$ with an \textit{odd} $n_r$ form the transversal $T_{F''}$ to $F''$ in $A_0$, and this can be continued to a transversal $T_{D}$ to $D$ in the whole $F_0$.

For uniformity of our notations set $A=C=F_0$ and choose the trivial transversals $T_A=T_C=\{1\}$.
Let $\varphi$ be the homomorphism  sending $b,c$ to $b_{1},\;c^2$. Then $\varphi:A\to B$ is an isomorphism.
Similarly, letting $\psi$ to be the homomorphism  sending $b,c$ to $b,\,c^2$ we get the isomorphism $\psi:C\to D$.
Using these isomorphisms we define the HNN-extension:
$$
\Gamma = F_0 *_{\varphi,\psi} (t,s)
= \langle
b,c,t,s \;\mathrel{|}\;
b^t=b^c,\; b^s=b,\; c^t=c^s=c^2
\rangle.
$$
In the finitely presented group $\Gamma$ pick the finitely generated subgroup $L=\langle
b_1, t, s
\rangle$, and prove that $F_0 \cap L = B_0$. One side of inclusion is evident and we should show $F_0 \cap L \le B_0$.
Using four transversals $T_A, T_B, T_C, T_D$ mentioned above we can bring any $x\in L$ (which is some product of elements $b_1^{\pm 1}$, $t^{\pm 1}$, $s^{\pm 1}$) to a normal form in $\Gamma$
by scanning $x$ from right to left by four rules:

\noindent
replace in $x$ a subword of type $t^{-1} w$ (with $w$ in \eqref{EQ the worm of w OLD}) by $t^{-1} b_{n_1}^{\varepsilon_1}\!\cdots b_{n_s}^{\varepsilon_s} \,t \cdot t^{-1} \!= b_{2n_1+1}^{\varepsilon_1}\!\cdots b_{2n_s+1}^{\varepsilon_s} t^{-1}$\!\!,

\noindent
replace a subword of type $s^{-1} w$ by $s^{-1} b_{n_1}^{\varepsilon_1}\!\cdots b_{n_s}^{\varepsilon_s} \,s \cdot s^{-1} = b_{2n_1}^{\varepsilon_1}\!\cdots b_{2n_s}^{\varepsilon_s} s^{-1}$\!\!,

\noindent
replace a subword of type $tw$ by
$t\, b_{n_1}^{\varepsilon_1}\!\cdots b_{n_{r-1}}^{\varepsilon_{r-1}}
\, \,t^{-1} \cdot\, t \;
b_{n_r}^{\varepsilon_r}\!\cdots b_{n_s}^{\varepsilon_s}
= \varphi^{-1}(b_{n_1}^{\varepsilon_1}\!\cdots b_{n_{r-1}}^{\varepsilon_{r-1}})$
$ \cdot t \;
b_{n_r}^{\varepsilon_r}\!\cdots b_{n_s}^{\varepsilon_s}
$ $
=b_{(n_1 - 1)/2}^{\varepsilon_1}\!\cdots b_{(n_{r-1}-1)/2}^{\varepsilon_{r-1}}
\cdot t\,
b_{n_r}^{\varepsilon_r}\!\cdots b_{n_s}^{\varepsilon_s}
$,
which is doable, as $n_1,\ldots,n_{r-1}$ are odd,

\noindent
replace a subword of type $sw$ by
$s\, b_{n_1}^{\varepsilon_1}\!\cdots b_{n_{r-1}}^{\varepsilon_{r-1}}
\, \,s^{-1} \cdot \,s \;
b_{n_r}^{\varepsilon_r}\!\cdots b_{n_s}^{\varepsilon_s}
= \psi^{-1}(b_{n_1}^{\varepsilon_1}\!\cdots b_{n_{r-1}}^{\varepsilon_{r-1}})$
$ \cdot s \;
b_{n_r}^{\varepsilon_r}\!\cdots b_{n_s}^{\varepsilon_s}
$ $
=b_{n_1/2}^{\varepsilon_1}\!\cdots b_{n_{r-1}/2}^{\varepsilon_{r-1}}
\cdot s\,
b_{n_r}^{\varepsilon_r}\!\cdots b_{n_s}^{\varepsilon_s}
$,
which is doable, as $n_1,\ldots,n_{r-1}$ are even this time.

Applying these four actions to $x$, we never get a \textit{new} $b_i^{\varepsilon_i}$ with a \textit{negative} $i$, i.e.,  
the normal form of $x$ consists of some factors $b_i^{\varepsilon_i}$ with non-negative $i$, and of some $t^{\pm 1}$\!\!,\; $s^{\pm 1}\!\!$.
If, moreover, $x$ is in $F_0$, then by uniqueness of the normal form we rule out such $t^{\pm 1}$\!\!,\; $s^{\pm 1}\!\!$,\, and get $g \in B_0$. 
\end{Example}

\begin{Remark}
The reader may compare the above proof with Lemma 3.8 in~\cite{Higman Subgroups of fP groups}. We show that $B_0$ is benign without using Lemma~3.1--Lemma~3.4 of~\cite{Higman Subgroups of fP groups}. Instead, we deduce it from \ref{SU The conjugates collecting process} and from from some simple combinatorics on words.  
\end{Remark}

\begin{Example}
\label{EX Uts is benign}
The subgroups $B_{l}=\langle b_i \mathrel{|}  i\ge l \rangle 
$ and $B_{k,l}=\langle b_i \mathrel{|} i <k \text{ or } i\ge l \rangle 
$ are benign in $F_0$ for any fixed integers $k\le l$.
Let $\mu(x)=x^c$ be  inner automorphism on $F_0$.
Since $\mu^l(b_i)=b_i^{c^{l}}=b_{i+l}$, it is clear that $\mu^l(B_0)=B_l$, and $B_l$ is benign 
by previous example and by repeated application of 
\eqref{PO 1 LE benign subgroup has benign image} in
Lemma~\ref{LE benign subgroup has benign image}.
$F_0$ has a homomorphism $\chi$ sending $b,c$ to $b,c^{-1}$\!.\, Then the image $ \chi(B_0)=\langle b_i \mathrel{|} i \le 0  \rangle$ and the image
$\mu^{k-1}\big(\chi(B_0)\big)=\langle b_i \mathrel{|} i \le k-1  \rangle$
both are benign in $F_0$.
Since $B_{k,l}$ is a free group, we have
$B_{k,l}= \big\langle 
\mu^l(B_0),\; \mu^{k-1}\big(\chi(B_0)\big)
\big\rangle$, and this subgroup also is benign by 
\eqref{PO 2 LE benign subgroups intersection and union are benign} in 
Lemma~\ref{LE benign subgroups intersection and union are benign}.
\end{Example}


To put some more generators into the game setup the free group $F^*=\langle 
a,b,c,\; g,h,k \rangle$. 
In analogy to $b_i$ and $b_f$ introduce $h_i=h^{k^i}$ for $i\in \Z$ and $h_f = h_0^{f(0)}\cdots h_{m-1}^{f(m-1)}$ along with $g_f=g^{h_f}$ for  $f\in \E_m$. 
Using the elements $a, b_i, g, h_i$ we can choose new free subgroups of $F^*$, such as 
$F_1\!=\!\langle 
g,\,
h_0, \ldots, h_{m-1}
\rangle$. 
Below we are going to construct some benign subgroups inside such subgroups. To make discussion simpler, we treat a group of type $F_1$ as a free group on generators $g,\,
h_0, \ldots, h_{m-1}$, ``forgetting'' the fact that $h_i$ is the conjugate $h^{k^i}$ for some $k$. In such groups we still can use the notations like $b_i$, $b_f$, $a_f$, $h_i$, $h_f$, $g_f$.
The reason of this manner of notations is to make the main proof in Section~\ref{SE The proof of Theorem} clearer.

\begin{Example} 
\label{EX E_m is benign}
The subgroup 
$F_{\E_m}\!=
\langle 
g_f \mathrel{|} f\in \E_m
\rangle
$ is benign in  $F_1\!=\!\langle 
g,\,
h_0, \ldots, h_{m-1}
\rangle$.
Indeed, for each $r=0,\ldots, m-1$ there is an isomorphism $\varphi_r$ on $F_1$ sending
$g,\,
h_0, \ldots, h_{r-1},  h_{r} , \ldots, h_{m-1}$
to 
$g^{h_r}\!,\,
h_0^{h_{r}}\!, \ldots, h_{r-1}^{h_r},  h_{r} , \ldots, h_{m-1}$.
Then $\Gamma = F_1 *_{\varphi_0 ,\ldots, \varphi_{m-1}} \!(t_0 ,\ldots, t_{m-1})$ is finitely presented.
The action $g_f^{ t_r^{\pm1}}\!\!=g_{f_r^{\pm}}$ in $\Gamma$ is trivial to verify. Say, for $f=(2,5,3)$ and $r=1$ we have:
$$
\big(g^{h_0^2 \, h_1^5 \, h_2^3}\big)^{t_1} \!\!=
\big(h_2^{-3}  h_1^{-5} h_0^{-2}  \big)^{t_1} \! g^{t_1}\,
\big( h_0^2 \, h_1^5 \, h_2^3\big)^{t_1}
\!\!=
h_2^{-3} 
h_1^{-5}
\big(h_0^{-2}\big)^{h_1}
g^{h_1} \,
\big(h_0^{2}\big)^{h_1}  h_1^{5}
h_2^3
=
g^{h_0^2 \, h_1^{5+1} \, h_2^3}
\!\!=
g^{h_{f_1^+}}\
$$
with $f_1^+\!\!=\!(2,6,3)$.
Thus, $L= \langle 
g,\,
t_0, \ldots, t_{m-1}
\rangle$ contains  the elements $g_f=g^{t_{0}^{f(0)}\cdots \,t_{m-1}^{f(m-1)}}$ for \textit{all} $f\in \E_m$.
On the other hand, 
applying \ref{SU The conjugates collecting process}
to any word $w\in L$ for
$\X=\{g\}$
and $\Y=\{t_0, \ldots, t_{m-1}\}$
we rewrite it as $w=u\cdot v$, where $u$ is a product of some conjugates of $g$ by some words in $t_i$ (and they are equal to some $g_f$ as we just saw), while $v$ is another word in letters $t_i$. 
If $w$ also is in $F_1$, then $v=1$ by uniqueness of the normal form, and so $F_1 \cap L \subseteq F_{\E_m}$.
\end{Example}

\begin{Example}
If $v_f=g_{f^+} \, b^{-1}_{m-1}  g_f^{-1}$, then the subgroup 
$V_{\E_m}\!\!=
\langle 
v_f \mathrel{|} f\in \E_m
\rangle
$ is benign in $F_2\!=\!\langle 
b_{m-1},\, 
g,\,
h_0, \ldots, h_{m-1}
\rangle$.
Let $\varphi_r$ be the isomorphism on $F_2$ 
fixing $b_{m-1}$, and with same effect on other generators as in previous example. 
$\Delta\!= \! F_2 *_{\varphi_0 ,\ldots, \varphi_{m-1}} \!(t_0 ,\ldots, t_{m-1})$ is finitely presented. The equality 
$v_f^{t_r^{\pm1}}\!\!=v_{f_r^\pm}^{\vphantom{1}}$ is trivial to verify in $\Delta$. \;Say, for $f=(2,5,3)$ and $r=1$:
$$
v_f^{t_1}=\!
\big(g^{h_0^2 \, h_1^5 \, h_2^{3+1}}\big)^{t_1} 
(b_{m-1}^{-1})^{t_1}
\big(g^{-\,h_0^2 \, h_1^5 \, h_2^{3}}\big)^{t_1}
=
g^{h_0^2 \, h_1^{5+1} \, h_2^{3+1}}
b_{m-1}^{-1}
g^{-b_0^2 \, h_1^{5+1} \, h_2^{3}} 
=
v_{f_1^+}\,.
$$
\vskip-1mm
\noindent
If $v_{(0)}\!=
g^{h_{m-1}}\, b_{m-1}^{-1}  g^{-1}
$\!\!,\;  then the above shows that 
$L= \langle 
v_{(0)},\, t_0, \ldots, t_{m-1}
\rangle$ contains  
$v_f\!=\!
v_{(0)}^{t_{0}^{f(0)}\cdots \,t_{m-1}^{\,f(m-1)}}$\! for \textit{all} $f\in \E_m$.
On the other hand, 
applying \ref{SU The conjugates collecting process}
to $w\in L$ for
$\X\!=\!\{v_{(0)}\}$
and $\Y\!=\!\{t_0, \ldots, t_{m-1}\}$ we like in previous example get 
$F_2 \cap L \subseteq V_{\E_m}$. I.e., $V_{\E_m}$ is benign. 
\end{Example}

\begin{Example}
\label{EX benign a b_i c_i}
If $z_f=g_f^{\vphantom8} b_f^{-1}$\!\!,\; then the subgroup $
Z_{\E_m}=
\langle \,
z_{f} \!\mathrel{|} f\in \E_m
\rangle
$ is benign in $F_3=\langle 
b_0, \ldots, b_{m-1},\, 
g,\,
h_0, \ldots, h_{m-1}
\rangle$.
First verify that 
$\langle Z_{\E_{m-1}}, V_{\E_m} \rangle = Z_{\E_m}$.
We have $V_{\E_m}\! \le Z_{\E_m}$ because
$
v_f=g_{f^+} \, b_{m-1}^{-1} \, g_f^{- 1}
=
g_{f^+}^{\vphantom8} \, b_{f^+}^{-1} \cdot b_f^{\vphantom8} g_f^{- 1}=
z_{f^+}^{\vphantom8}\!\! \cdot z_{f}^{-1} \in Z_{\E_m}$.
And since also $Z_{\E_{m-1}}\! \le Z_{\E_{m}}$, we have $\langle Z_{\E_{m-1}}, V_{\E_{m}} \rangle \le Z_{\E_{m}}$. Next, show by induction on $f(m-1)$ that $z_f\in \langle Z_{\E_{m-1}}, V_{\E_{m}} \rangle $. 
When $f(m-1)=0$, then $z_f\in Z_{\E_{m-1}}$.
And, since then $f^+(m-1)=1$, we from 
$v_f^{\vphantom{+}}=
z_{f^+}^{\vphantom{+}}\! \cdot z_f^{-1} \in V_{\E_{m}}$  get $z_{f^+}\in \langle Z_{\E_{m-1}}, V_{\E_{m}} \rangle$.
Say, for $f=(2,5,0)$ we have
$
Z_{\E_{3-1}}\!
\ni
v_{(2,5,0)}= z_{(2,5,0+1)}^{\vphantom{+}} \cdot z_{(2,5,0)}^{-1}$, and so
$
z_{(2,5,0+1)} \in \langle Z_{\E_{3-1}}, V_{\E_{3}} \rangle
$.
The cases with $f(m\!-\!1)=1,2,\ldots$ are covered similarly. Symmetric arguments cover the case with negative $f(m\!-\!1)$. Since  $Z_{\E_{m}}=\langle Z_{\E_{m\!-\!1}}, V_{\E_{m}} \rangle$, the fact that $Z_{\E_{m}}$ is benign follows by induction: $Z_{\E_{0}} =\langle g_{(0)}^{\vphantom8} b_{(0)}^{-1} \rangle=\langle g \rangle$ is benign, as it is finitely generated.
If $Z_{\E_{m\!-\!1}}$ is benign, then $Z_{\E_{m}}=\langle Z_{\E_{m\!-\!1}}, V_{\E_{m}} \rangle$ is benign by 
\eqref{PO 2 LE benign subgroups intersection and union are benign} in 
Lemma~\ref{LE benign subgroups intersection and union are benign} and by previous example.
\end{Example}


\section{The modified proof of Theorem~\ref{TH Higman 4}}
\label{SE The proof of Theorem}

\noindent
Denote by 
$\BB$ 
the set of all subsets $\B \subseteq \E$ for which $A_{\B}= \langle a_f \;|\; f \in \B\, \rangle$ is benign in $F=\langle a,b,c \rangle$.
First show that $\Zz,\S\in \BB$, and then verify that if a subset $\B$ is obtained from $\Zz$ and/or $\S$ by Higman operations \eqref{EQ Higman operations}, then $\B$ also is in $\BB$. 
This proves Theorem~\ref{TH Higman 4} as by Theorem~3 in \cite{Higman Subgroups of fP groups} every recursively enumerable subset $\B\subseteq \mathcal E$ can be obtained that way.

\smallskip
The free group
$B\!=\!\langle a,b_i \mathrel{|} i\!\in \!\Z \rangle$  is a product $Z_0*Z_1$ for 
$
Z_0\!=\!\langle b_i \mathrel{|} i> 0 \rangle 
$ and $
Z_1=\langle a,b_i \mathrel{|} 
i \le  0 \rangle.
$
Of these $Z_0=B_1$ is benign by Example~\ref{EX Uts is benign}, 
and $Z_1$ is benign by 
Lemma~\ref{LE benign subgroup has benign image},
Lemma~\ref{LE benign subgroups intersection and union are benign}
and  Example~\ref{EX Uts is benign}, since $
Z_1=\big \langle \langle a\rangle,\, \langle b_i \mathrel{|} 
i \le  0 \rangle \big \rangle
$ with the benign subgroup  $\langle b_i \mathrel{|} 
i \le  0 \rangle=\chi\big(\langle b_i \mathrel{|} 
i \ge  0 \rangle \big)$ for the 
homomorphism $\chi$ sending $b,c$ to $b,c^{-1}$ defined in Example~\ref{EX Uts is benign}. 
So there are finitely presented  $K_0$ and $K_1$ which contain $F$, and which respectively have finitely generated subgroups $L_0$ and $L_1$ such that 
$F \cap L_0 = Z_0$ and 
$F \cap L_1 = Z_1$.
Then both HNN-extensions
$\Lambda_0 = K_0 *_{L_0} \!t$ and  $\Lambda_1 = K_1 *_{L_1} \!s$ also are finitely presented.

Amalgamating $F$ in $\Lambda_0$ and $\Lambda_1$ we get the  finitely presented group $\Theta\!=\!\Lambda_0 *_{F} \Lambda_1$ in which we can single out a smaller subgroup.
Namely, by \eqref{PO1 LE subgroups in HNN-extension} in
Lemma~\ref{LE subgroups in HNN-extension}
we in $\Lambda_0$ have $\langle F, t \rangle = F *_{F\, \cap\, L_0}\! t = F *_{Z_0} \! t$, and this HNN-extension contains the subgroup $F *_{Z_0} \! F^t$.
In the same way in $\Lambda_1$ we get
$\langle F, s \rangle  = F *_{Z_1} \!s$
containing 
$F *_{Z_1} \!F^s$\!.
And since $F *_{Z_0} \!F^t$ and  $F *_{Z_1} \!F^s$ both contain $F$, we by \eqref{PO 1 in CO NEW G*H lemma corollary} in Corollary~\ref{CO NEW G*H lemma corollary} have 
$\langle F,\, F^t, F^s\rangle=
 (F *_{Z_0} \!F^t) *_F
 (F *_{Z_1} \!F^s)
$. 
The advantage of the constructed smaller subgroup  is that its defining relations are trivial to list: 
the relations identifying $Z_0$ with $Z_0^t$, and the relations identifying $Z_1$ with $Z_1^s$. These identifications can be understood as two isomorphisms, and since 
$Z_0$ and $Z_1$ generate their \textit{free} product, both isomorphisms can be continued to a single isomorphism $\nu:Z_0 * Z_1 \to \langle Z_0^t, Z_1^s\rangle$ and, hence, $\langle F,\, F^t, F^s\rangle$ is noting but the amalgamated free product of $F *_{\nu } \langle F^t, F^s\rangle$. 
Since all the above listed relations of $\langle F,\, F^t, F^s\rangle$ are covered by $\nu$, there are no relations left for $\langle F^t, F^s\rangle$, i.e.,  $\langle F^t, F^s\rangle = F^t * F^s$.

Therefore, the identical automorphism on $F^t$ and the automorphism given by conjugation by $b^s$ on $F^s$ possess a common extension $\omega$ on $\langle F^t, F^s \rangle$. The respective HNN-extension
$
\Psi = \Theta *_{\omega}\! d 
$
is finitely presented as $\omega$ can be defined by its values on just six generators $a^t\!\!, b^t\!\!, c^t\!\!,\; a^{\,s}\!\!, b^s\!\!, c^s$ of $F^t * F^s$.
For the generators of $B$ we in $\Psi$ have:
\begin{equation}
\label{EQ action of d on F}
\text{
$b_i^d= b_i$ for $i> 0$,\quad
$a^d = a^b$\!\!,\;\quad
$b_i^d= b_i^b$ for $i \le 0$
}
\end{equation}
(say, $a^d = (a^s)^\omega
= (a^s)^{b^s}=s^{-1} b^{-1} s \;\cdot\; s^{-1} a\, s \;\cdot\; s^{-1} b\, s
= (a^b)^{\,s} = a^b$, since $a^b \in Z_1^{\vphantom{1}} \!= Z_1^s$; \; or for $i \le 0$ we have $
b_i^d 
= (b_i^s)^\omega
= (b_i^s)^{b^s}
= s^{-1} b^{-1} s \;\cdot\; s^{-1} b_i\, s \;\cdot\; s^{-1} b\, s
= (b_i^b)^{\,s}
= b_i^b
$, since $b_i^b \in Z_1^{\vphantom{1}}$).

Next, take $\delta$ to be the automorphism on $F$ which sends $a,b,c$ to $a,b^c\!\!,\,c$. In the finitely presented group $\Delta  = \Psi  *_\delta e$ denote $d_i = d^{e^i}$ for $i\in \Z$.
From \eqref{EQ action of d on F}, from the action of $e$ on $a,b,c$ and from $b_i^{e^j}
= 
(c^{-i})^{e^j} b^{e^j} (c^{i})^{e^j}
= 
c^{-i} b^{c^j} c^{i}
=b_{i+j}$ it follows:
\begin{equation}
\label{EQ action of d_i on F}
\text{
$b_i^{d_j}= b_i$ \,for\, $i> j$,\quad
$a^{d_j} = a^{b_j}$\!\!\!\!,\;\;\quad
$b_i^{d_j}= b_i^{b_j}$ \,for\, $i\le j$
}
\end{equation}
(say, $a^{d_j}=a^{e^{-j} d\, e^j}=a^{d\, e^j}=(b^{-1}a\, b)^{e^j}=b^{-c^j} a b^{c^j}=a^{b_j}$; \; or for $i\le j$ we have $b_i^{d_j}=(c^{-i} b\,  c^i)^{\,e^{-j}d e^j}
=(c^{-i} b^{c^{-j}}  c^i)^{\,d\, e^j}
=b_{i-j}^{\,d e^j} 
=b_{i-j}^{\,b e^j}
=(b^{-1} b_{i-j} \, b)^{\, e^j} 
=b^{-c^j} \cdot 
c^{-(i-j)} b^{c_j} c^{(i-j)}
\cdot b^{c^j}
=b_i^{b_j}
$).

Here is the main feature for the sake of which $\Delta$ was built:

\begin{Lemma} 
\label{LE action of d_m on f}
For any $f \in \mathcal E$ and for any fixed integer $j$ 
we have 
$
a_f^{d_j} =\! a_{f_{j}^+}$ 
and
$
a_f^{\,d_j^{-1}}\!\! = a_{f_{j}^-}
$.
\end{Lemma}

\begin{proof} 
Clearly, $a_f^{d_j} 
\!=\! b_f^{\!-\,d_j}\,
 a^{d_j} 
b_f^{\,d_j}
$\!. By \eqref{EQ action of d_i on F} the effect of conjugation of factors $b_{i}^{f(i)}$ \!in $b_f$ by $d_j$ is so that in $b_f$ all the factors standing \textit{before} $b_{j+1}^{f(j+1)}$ are conjugated by the \textit{same} $b_j$, while other factors are not changed. After cancellations of all $b_j^{\vphantom8} b_j^{-1}$ in and between the factors of $a_f^{\!-b_j}
$\!,\, $a^{b_j}$ and $
a_f^{b_j}$ we get the needed equality. 
The second equality is considered similarly.
\end{proof}

\begin{Example}
\label{EX d acts on a_f} 
Let $j=1$ and $f=(2,5,3)$.
Then $b_f=b_0^{2}\,b_1^5\,b_2^{3}$, and by \eqref{EQ action of d_i on F} we have:
$
a_f^{d_1}=
\big(
b_2^{\!-3}b_1^{\!-5}b_0^{\!-2}
\; a\;
b_0^{2}b_1^{5}b_2^{3} \,
\big)^{d_1}
\!\!= 
b_2^{-3}\,
\big(b_1^{\!-1} b_1^{\!-5} b_1^{\vphantom8}\big)
\big(b_1^{\!-1} b_0^{\!-2} b_1^{\vphantom8}\big)
\big(b_1^{\!-1} a b_1^{\vphantom8}\big)  
\big(b_1^{\!-1} b_0^{2}b_1^{\vphantom8}\big)\,
\big(b_1^{\!-1} b_1^{5}b_1^{\vphantom8}\big)
\,b_2^{3} 
$,
which after cancellations is equal to
$
b_2^{\!-3}b_1^{\!-\,6}  b_0^{\!-2}
\;a\;
b_0^{2}b_1^6b_2^{3}
= a_{f_{1}^+}
$
with
$f_{1}^+ = (2,6,3)$.
\end{Example}

\begin{Lemma}
\label{LE X_Z is benign}
$\BB$ contains the sets $\Zz$ and 
$\S$.
\end{Lemma}

\begin{proof}
Since  $\Zz$  consists of $f\!=\!(0)$ only, then $A_{\Zz}=\langle a \rangle$. Being finitely generated $\langle a \rangle$ is benign.

$\S$ contains the function $f=(0,1)$.
Apply Lemma~\ref{LE action of d_m on f} to $a_{f}$  repeatedly we get:
\begin{equation}
\label{EQ d0d1 acts on a_01}
a_{f}^{(d_0 d_1)^n}
\!\!=\big(a_{(0,1)}^{d_0 }\big)^{d_1  \, (d_0 d_1)^{n-1}}
\!\!\!\!=a_{(0+1,\,1)}^{d_1  \, (d_0 d_1)^{n-1}}
\!\!\!=a_{(1,\,1+1)}^{(d_0 d_1)^{n-1}}
\!\!\!=a_{(2,3)}^{(d_0 d_1)^{n-2}}
\!\!\!= \cdots 
=a_{(n,n+1)}.
\end{equation}
Thus, $a_{(n,n+1)} \in \langle
a_{f}, d_0 d_1
 \rangle $ for any $n\in \Z$, and so $A_{\S}\le \langle
 a_{f}, d_0 d_1
 \rangle$.
Using \eqref{EQ elements from <x,y>} 
for $x=a_{f}$
and
$y= d_0 d_1$,
we can rewrite any element $w\in \langle
a_{f}, d_0 d_1
\rangle$
as $w=u\cdot v$, where $u$ is a product of some conjugates $a_{f}^{\pm (d_0 d_1)^{n_i}}$\!\!\!\!,\, and $v=(d_0 d_1)^k$\!.
By \eqref{EQ d0d1 acts on a_01} all those conjugates are in $F$.
Thus, if also $w \in F$, then  $v\in F$.
Let us show this may happen for  $v=1$ only.
Since $d\notin F$, we can include $d$ in any transversal $T_F$ to $F$ in $\Psi$.
Then $v = (d^{e^0} d^{e^1})^k  
= (d\, e^{-1} \! d\, e)^k$ is in normal form in $\Delta$  (with respect to $T_F$), i.e., $k=0$, if $v\in F$.
So $F \cap \langle
a_{f}, d_0 d_1
\rangle = A_{\S}$, and  $A_{\S}$ is benign.
\end{proof}

\begin{Theorem}
\label{LE closed under higman operations}
The set $\BB$ is closed under the Higman operations \eqref{EQ Higman operations}.
\end{Theorem}

This fact at once follows from Lemma~\ref{LE benign subgroups intersection and union are benign} for binary operations $\iota$ and $\upsilon$.
The proofs for unary operations
$\rho,\, \sigma,\, \tau,\, \theta,\, \zeta,\, \pi,\, \omega_m$ follow from specific cases and examples below. In each of them we suppose $\B $ is in $\BB$, and show that a Higman operation keeps it in $\BB$.

\smallskip

\textit{$\BB$ is closed under  $\rho$} because for the automorphism $\varphi$  of $F$  
sending $a,b,c$ to $a,b,c^{-1}$ 
we have 
$\varphi(b_i)=\varphi(b^{c^i})=b^{c^{-i}}=b_{-i}$. So 
$A_{\rho(\B)} =  \varphi(A_{\B})$ is benign by 
\eqref{PO 1 LE benign subgroup has benign image} in 
Lemma~\ref{LE benign subgroup has benign image}.

\smallskip

\textit{$\BB$ is closed under  $\sigma$.} 
Indeed, for the automorphism $\varphi$  of $F$  
sending $a,b,c$ to $a,b^{c}\!\!,\,c$ 
we have 
$\varphi(b_i)=\varphi(b^{c^i})=b^{c^{i+1}}\!=b_{i+1}$, and so 
$A_{\sigma(\B)} =  \varphi(A_{\B})$ is benign by 
\eqref{PO 1 LE benign subgroup has benign image}
in Lemma~\ref{LE benign subgroup has benign image}.

\textit{$\BB$ is closed under  $\zeta$.}
For each $f \in \B$ we 
in $\Delta$ 
by Lemma~\ref{LE action of d_m on f} have
$a_f^{d^k}\!\!=a_{f'}^{\vphantom8}$, with
$f'(0)\!=\!f(0)+k$, and $f'(i)\!=\!f(i)$ for $i\!\neq\! 0$. 
Any $f'\!\!\in\! \zeta(\B)$ can be obtained this way, and so $R\!=\!\langle\, A_{\B}, d\rangle$ contains $A_{\zeta(\B)}$.
On the other hand, for any $w \!\in\! R$ we apply  \ref{SU The conjugates collecting process}  for
$\X=\{a_f \mathrel{|} f \in \B\}$
and $\Y=\{d\}$ to get $w = u \cdot v$, where $u$ is a product of factors of type $a_f^{\pm d^{k_i}}\!\!\in A_{\zeta(\B)}$, and where $v=d^k$\!. 
If also $w \in F$, then $v\!=\!1$ by uniqueness of  normal form, i.e., $w\in A_{\zeta(\B)}$ and $F \!\cap R= A_{\zeta(\B)}$. 
Since $R$ is benign by 
\eqref{PO 2 LE benign subgroups intersection and union are benign} in
Lemma~\ref{LE benign subgroups intersection and union are benign}, $ A_{\zeta(\B)}$ is benign by 
\eqref{PO 1 LE benign subgroups intersection and union are benign} in
Lemma~\ref{LE benign subgroups intersection and union are benign}.

\vskip1mm

\textit{$\BB$ is closed under  $\pi$.}
In previous case take $R=\langle\, A_{\B}, d_1, d_2,\ldots\, \rangle$ instead. 
$R$ contains $A_{\pi(\B)}$, and  applying \ref{SU The conjugates collecting process} for
$\X=\{a_f \mathrel{|} f \in \B \}$
and $\Y=\{d_1, d_2,\ldots\}$ we get 
$F \cap \,R= A_{\pi(\B)}$.
Since $Z_0$
is benign in $\langle b,c \rangle$ by Example~\ref{EX Uts is benign}, its image  $\langle d_0, d_1,\ldots \rangle$  
together with the subgroup 
$\langle d_1, d_2,\ldots \rangle$ 
are benign in $\langle d,e \rangle$ and, thus, in $\Delta$.  Hence $R$ is benign by
\eqref{PO 2 LE benign subgroups intersection and union are benign} in
Lemma~\ref{LE benign subgroups intersection and union are benign}. Then 
$A_{\pi(\B)}$ is benign by \eqref{PO 1 LE benign subgroups intersection and union are benign} in
Lemma~\ref{LE benign subgroups intersection and union are benign}.

\vskip1mm

\textit{$\BB$ is closed under  $\theta$.}
Modify the idea of previous points. The subgroup
$\langle d_{i} \mathrel{|} 
\text{all odd $i$}\, \rangle$
is benign as it is the isomorphic image of the benign subgroup $A_{2,1}$ from Example~\ref{EX b SOME powers is benigh}. 
Since $A_{\B}$ is benign,  $R_1=\langle \, A_{\B},  d_{i} \mathrel{|} 
\text{all odd $i$}\, \rangle$ also is benign.
Applying \ref{SU The conjugates collecting process}  for 
$\X=\{a_f \mathrel{|} f \in {\B} \}$
and $\Y=\{ d_{i} \mathrel{|} 
\text{all odd $i$}\, \}$ we get that
$F \cap R_1= A_{{\B}_1}$, where $f' \in {\B}_1$
if and only if there is an $f \in {\B}$ such that $f'(i)=f(i)$ for all even $i$ (there can be any values $f'(i)$ for odd $i$).
Similarity, $\langle d_{i} \mathrel{|} 
\text{all even $i$}\, \rangle$
 is benign (use $A_{2,0}$ this time).
Thus, $R_2=\langle \, a,  d_{i} \mathrel{|} 
\text{all even $i$}\, \rangle$ also is benign.
Applying \ref{SU The conjugates collecting process} for 
$\X=\{a\}$
and $\Y=\{ d_{i} \mathrel{|} 
\text{all even $i$}\, \}$ we get that
$F \cap R_2= A_{\B_2}$ where $f' \in \B_2$
if and only if $f'(i)=0$ for all odd $i$ (there can be any values $f'(i)$ for even $i$).
Then  $A_{\B_1}\! \cap A_{\B_2}$ is benign.
Since the $a_f$ are \textit{free} generators, 
$A_{\B_1}\! \cap A_{\B_2}\!\!=\!A_\Y$ where $\Y\!=\!\B_1 \!\cap \B_2$.
But $\B_1 \!\cap \B_2$
clearly consists of all those functions $f'$ which coincide with some $f \in \B$ on all even $i$, and are zero elsewhere.
Let $\gamma$ be the homomorphism sending sending $a,b,c$ to $a,b,c^2$.
Then $A_{\theta(\B)}=\gamma^{-1}(A_\Y)$, and so $A_{\theta(\B)}$ is benign by \eqref{PO 2 LE benign subgroup has benign image}
in Lemma~\ref{LE benign subgroup has benign image}.

\vskip1mm

\textit{$\BB$ is closed under  $\tau$.}
By Example~\ref{EX Uts is benign} 
\, $B_{0,2}$ is benign in $F$. Take $K$ to be a finitely presented group with a finitely generated subgroup $L$ such that $F \cap L = B_{0,2}$, and set   $\Lambda = K *_{L} \! t$.
As it is easy to check 
$\langle K, K^t \rangle 
$ is equal to $ K *_{L} K^t$. 
Since $B \cap B_{0,2}  = B_{0,2}= F^t \cap B_{0,2}$, we by 
\eqref{PO 1 in CO NEW G*H lemma corollary} in Corollary~\ref{CO NEW G*H lemma corollary}
have $\langle B, F^t \rangle = B *_{B_{\,0,2}} \!\!F^t$ in $\Lambda$. 
Since $B$ is free, it has an automorphism swapping
$b_0$ and $b_1$, while fixing $a$ and $b_i$
for any $i\neq 0,1$.
This automorphism and the  identical automorphism of $F^t$ agree on $B_{0,2}$, so they have a common extension $\xi$ on entire $B *_{B_{\,0,2}} \!\!F^t$\!. Then $\Phi= \Lambda *_{\xi} q$ is finitely presented as $\xi$ 
can be defined by its values on just six generators $a, b_0, b_1;\; a^t\!, b^t\!, c^t$ of $B *_{B_{0,2}} \!F^t$ (when $i\neq 0,1$, then $b_i = b_i^t\in F^t$).
For every $f\! \in \mathcal E$ we clearly have $\tau (b_f) = b_{f'}$ where
$f'(0)=f(1)$,\,
$f'(1)=f(0)$,\,
and
$f'(i)=f(i)$,\,
for $i\neq 0,1$.
Thus,  $A_{\tau(\B)}=A_{\B}^q$ in $\Phi$.
Applying 
\eqref{PO 1 LE benign subgroup has benign image} in
Lemma~\ref{LE benign subgroup has benign image} for the trivial embedding $\alpha: F \to \Phi$ we get that $A_{\B}$ is benign in $\Phi$. By the above construction and again by \eqref{PO 1 LE benign subgroup has benign image} in
Lemma~\ref{LE benign subgroup has benign image} $A_{\tau(\B)}=A_{\B}^{\,q}$ also is benign in $\Phi$.
The preimage $\alpha^{-1}(A_{\tau(\B)})$ in $F$ clearly is $A_{\tau(\B)}$, which then is benign in $F$ by \eqref{PO 2 LE benign subgroup has benign image} in
Lemma~\ref{LE benign subgroup has benign image}.

\vskip1mm

\textit{$\BB$ is closed under  $\omega_m$ for each $m=1,2,\ldots$}\;
%
If $\B_m  = \iota(\B, \E_m)$, then $A_{\B_m}$\! is benign as  $A_{\E_m}$ is benign by Example~\ref{EX E_m is benign}. 
Clearly, $\omega_m(\B)=\omega_m(\B_m)$, and we may suppose $\B = \B_m$.
Since $B_{0,m}$ is benign
in $\langle b,c\rangle$ by Example~\ref{EX Uts is benign},  $\langle b,c\rangle$ is embeddable into some finitely presented group $K$ with a finitely generated subgroup $P$ such that 
$\langle b,c\rangle \cap P = B_{0,m}$. 
Set
$
\Sigma = K *_P (g,h,k)
$, where each of $g,h,k$ stabilizes $P$\, (and also $B_{0,m}$).

Since the intersection of $\langle b_0,\ldots, b_{m-1} \rangle$ with $B_{0,m}$ and, thus, with $P$ is trivial, by
\eqref{PO 1 in CO NEW G*t lemma corollary} in Corollary~\ref{CO NEW G*t lemma corollary} (see also Remark~\ref{RE about multiple stable letters})
the elements $b_0,\ldots, b_{m-1},\, g,h,k $ generate in $\Sigma$ the free subgroup $\langle b_0,\ldots, b_{m-1}\rangle * \langle g,h,k \rangle$. 
Denoting $h_i=h^{k^i}$ for $i \in \Z$,
we get another free subgroup 
$\langle b_0,\ldots, b_{m-1},\, g,\, h_0,\ldots, h_{m-1} \rangle$ in $\Sigma$.
Setting $
h_f=h_0^{f(0)} \!\cdots\, h_{m-1}^{f(m-1)}
$ and $g_f=g^{h_f}$ we
by Example~\ref{EX benign a b_i c_i} get a benign subgroup 
$L=\langle g_f^{\vphantom8} b_f^{-1} 
\mathrel{|}
f\in \mathcal E_m \rangle$. I.e., $\Sigma$ is embeddable into a finitely presented  $M$ with a finitely generated subgroup $R$ such that 
$\Sigma \cap R = L$. Take $\Pi=M *_{R} a$,
with some $a$ fixing $R$.  

If we show that $a,b,c$ generate a free subgroup in $\Pi$, we can identify it with our initial group $F=\langle a,b,c \rangle$.
That will follow from \eqref{PO 1 in CO NEW G*t lemma corollary} in Corollary~\ref{CO NEW G*t lemma corollary},
if we verify
$\langle a \rangle \cap R=\1$ (which is evident) and $\langle b, c \rangle \cap \,R=\1$.  By construction $\langle b, c \rangle \cap\, R=\langle b, c \rangle \cap L$.
Any two distinct words on letters $b_0,\ldots, b_{m-1}$ are distinct modulo $B_{0,m}$ and, thus, are distinct in $K$ modulo $P$. I.e., we can choose a transversal $T_P$ to $P$ in $K$ containing all the words on $b_0,\ldots, b_{m-1}$. Any non-trivial product $v$ of elements of type $g_f^{\vphantom8} b_f^{-1}\!=\! g^{h_0^{f(0)} \!\cdots\, h_{m-1}^{f(m-1)}} b_{m-1}^{-f(m-1)} \!\cdots b_{0}^{-f(0)}$ already is in normal form in $\Sigma$ (with respect to $T_P$) because it is a product of stable letters and transversal's elements. $v$ will be in $\langle b,c \rangle$ only if no stable letter is preset, i.e., if it is trivial.

Since  $F=\langle a,b,c \rangle$ is free, using its isomorphism $\rho$ sending 
$a, b, c$ to $a, b^{c^m}\!\!\!\!\!,\;c$ define 
$\Omega=\Pi *_{\rho} r$.
Denote $W_{\B}=\langle g_f\!,\, a,\, r \mathrel{|} f\in \B\rangle$. For any  $l\in \omega_m \B$ the element $a_l=a^{b_l}$ is in $W_{\B}$. Let us show this simple fact by a step-by-step construction example. 
Let $m=3$ and let  $(2,5,3)$, $(7,2,4) \in \B$.
Then $\omega_3 \B$ contains, say, 
$
l=(
0,0,0,\;
7,2,4,\;
0,0,0,\;
2,5,3,\;
7,2,4)
$. To show that $a^{b_l}\in W_{\B}$ start by the initial function $l_1=(7,2,4)$, and then modify it to achieve the function $l$ above.
The relation $(g_f^{\vphantom8} b_f^{-1})^a=g_f^{\vphantom8} b_f^{-1}$  is equivalent to $a^{g_f\vphantom{j}}\!=a^{b_f}$.

\vskip-1mm
\textit{Step 1.} Since  $f=l_1=(7,2,4)$ is in $ \B$, then $g_{l_1} \in W_{\B}$, and so 
$
a^{g_{l_1}}\!\!=
a^{b_{l_1}}=
a^{b_0^{7}\,b_1^{2}b_2^{4}}
\in W_{\B}$. 

\textit{Step 2.} 
Since 
$b_i^r=b_i^\rho\!=b_{i+3}$, conjugating the above element by $r$ we get that
$$
(a^{b_{l_1}})^r\!\!=
(a^r)^{(b_0^{7}\,b_1^{2}b_2^{4})^r}\!\!=
a^{b_3^{7}\,b_4^{2}b_5^{4}}
=
a^{b_0^{0}\,b_1^{0}b_2^{0} \;\cdot\; b_3^{7}\,b_4^{2}b_5^{4}}=
a^{b_{l_2}}\in W_{\B}
\text{\quad (with $l_2=(0,0,0,\;
7,2,4)$)}.
$$

Conjugate the above by 
$g_f$ for $f=(2,5,3)$. We have:
$
\big(\! a^{b_{l_2}}\big)^{g_f}
\!\!\!=\, 
a^{b_{l_2}\cdot \,g_f}
\!\!= a^{b_3^{7}\,b_4^{2}b_5^{4}\;\cdot \,g_f}\!
$.

\textit{Step 3.} Each of $g,h,k$ commutes with each of 
$b_i$ for $i <0$ or $i\ge m=3$, 
and so $g_f$ commutes with $b_3^{7}\,b_4^{2}b_5^{4}$, and 
$
a^{b_{l_2}\cdot \,g_f}
=
a^{g_f \cdot \, b_3^{7}\,b_4^{2}b_5^{4}}
$. 
Then again applying step 1 we get:
\vskip-3mm
$$
\big(\!a^{g_f}\!\big)^{b_3^{7}\,b_4^{2}b_5^{4}}
=
a^{b_0^{2}\,b_1^{5}b_2^{3} \;\cdot\; b_3^{7}\,b_4^{2}b_5^{4}}=
a^{b_{l_3}}\in W_{\B}
\text{\quad (with $l_3=(2,5,3,\;
7,2,4)$)}.
$$
Then we apply step 2 \textit{twice}, i.e., conjugate the above by $r^2$ to get 
the element $a^{b_{l_4}}$ with
$
l_4=(
0,0,0,\;
0,0,0,\;
2,5,3,\;
7,2,4)
$.
Next apply step 3 and step 1 again to conjugate 
$a^{b_{l_4}}$ by 
$g_f$ for $f=(7,2,4)$. We get the element 
$a^{b_{l_5}}$ with
$
l_5=(
7,2,4,\;
0,0,0,\;
2,5,3,\;
7,2,4)
$.
Finally, we again apply step 2, i.e., conjugate $a^{b_{l_5}}$  by $r$ to discover in $ W_{\B}$ the element $a^{b_{l}}=a_l$.
Since this procedure can easily be performed for an arbitrary $l \in \omega_m \B$, we get that $A_{\omega_m \B} \le W_{\B}$.

\vskip1mm
Let us single out an auxiliary ``slimmer'' subgroup $\Omega'$ in $\Omega$. 
Since $\langle b,c\rangle \cap P = B_{0,m}$, 
the subgroup $\big\langle \langle b,c\rangle, g,h,k \big\rangle$ is equal to $\Sigma' = \langle b,c\rangle *_{B_{0,m}} (g,h,k)$. 
Since $\Sigma'$ contains each of $b,c, g,h,k$ it also contains $L$. Thus, $\Sigma' \cap R = L$, and using 
\eqref{PO 1 in CO NEW G*t lemma corollary} in Corollary~\ref{CO NEW G*t lemma corollary}
we see that  $\langle \Sigma', a \rangle$ is equal to $\Pi'= \Sigma' *_L a$. 
Finally, since $a,b,c\in \Pi'$, the subgroup $ \langle \Pi', r \rangle$ is equal to $\Omega' = \Pi' *_\rho r$.
The advantage of $\Omega'$ is that it is an extension of $\langle b,c\rangle$ by means of three ``nested'' HNN-extensions, and we, thus, possess the \textit{full} list of defining relations of $\Omega'$.

Since any $w \in W_{\B}$ also is in $\Omega'$, it can be brought to its normal form involving $r$ and some elements from $\Pi'$.
The latters can in turn be brought to normal forms involving $a$ and some elements from $\Sigma'$. 
Then the latters can further be brought to normal forms involving $g, h,k$ and some elements from $\langle b,c\rangle$. 
I.e., $w$ can be brought to a \textit{unique ``nested'' normal form}.
Detect the cases when it involves $a,b,c$ only.
The only relations of $\Omega'$ involving $g,h,k$ are  $a^{g_f}\!=a^{b_f}$. Thus, the only way by which $g,h,k$ may be eliminated in the normal form  is to have in $w$ subwords of type $g_f^{\!-1} a\, g_f^{\vphantom8} = a^{g_f}$\!,\; to replace them by respective $a^{b_f}\in\langle a,b,c\rangle$. 
If after that some subwords $g_f$ still remain, then three scenario cases are possible: 

\textit{Case 1}. We may have a subword of type $w'=g_f^{-1}  a^{b_l}  g_f^{\vphantom8}$ for 
such an $l$ that $l(i)=0$ for $i=0,\ldots,m-1$.
Check example of step 1, when this is achieved for $l=l_2=(0,0,0,\;
7,2,4)$ and $f=(2,5,3)$. 
Then just replace $w'$ by $a^{b_{l_3}}$ for an $l_3\in \omega_m \B$ (e.g. $l_3=(2,5,3,\;
7,2,4)$). 

\textit{Case 2}. If  $w'=g_f^{-1}  a^{b_l}  g_f$,\, but the condition $l(i)=0$ fails for an $i=0,\ldots,m-1$, then $g_f$ does \textit{not} commute with $b_l$, so we cannot apply the relation $a^{g_f}\!=a^{b_f}$, and so $w \notin \langle a,b,c\rangle$.
Turning to  example in steps 1--3, notice that for, say, $f=(2,5,3) \in \B$ we may \textit{never} get something like
$a^{(g_f)^{\,2} }\!\!
=(a^{b_0^{2}\,b_1^{5}\,b_2^{3}})^{g_f}\!\!
=a^{(b_0^{2}\,b_1^{5}\,b_2^{3})^2}
$ because $g_f$ does not commute with $b_0,  b_1,  b_2$. That is, all the \textit{new} functions $l$ we get are from $\omega_m \B$ \textit{only}.

\textit{Case 3}. If $g_f$ is in $w$, but not in a subword $g_f^{-1}  a^{b_l}  g_f$, we again have $w \notin \langle a,b,c\rangle$.

This means, if $w \in \langle a,b,c\rangle$, then  elimination of  $g,h,k$  turns $w$ to a product of elements from  $\langle r\rangle$ and of some  $a^{b_l}$ for some $l\in \omega_m \B$ ($a$ also is of that type, as $(0)\in\B$).
Now apply  \ref{SU The conjugates collecting process}  for 
$\X=\{a^{b_l} \mathrel{|} l\in \omega_m \B \}$
and $\Y=\{r\}$ to state that
$w$ is a product of some power $r^k$ and of some elements each of which is an $a^{b_l}$ conjugated by a power $r^{n_i}$ of $r$.
These conjugates are in $\omega_m \B$, and so $w\in \langle a,b,c\rangle$ if and only if $k=0$, i.e., if 
$w\in A_{\omega_m \B}$.

It remains to notice that since $F=\langle a,b,c\rangle$ is a free group,  $F \cap W_{\B}=A_{\omega_m \B}$. 
Since $A_{\B}$ is benign, its image 
$\langle g_f \mathrel{|} f\in \B\rangle$ also is benign. Then 
$W_{\B}=\langle g_f\!,\, a,\, r \mathrel{|} f\in \B\rangle$ is benign by
\eqref{PO 2 LE benign subgroups intersection and union are benign} in Lemma~\ref{LE benign subgroups intersection and union are benign}, and  $A_{\omega_m \B}$ is benign by \eqref{PO 1 LE benign subgroups intersection and union are benign} in Lemma~\ref{LE benign subgroups intersection and union are benign}. 

\vskip-2mm

\medskip
\noindent 
E-mail:
\href{mailto:v.mikaelian@gmail.com}{v.mikaelian@gmail.com}
$\vphantom{b^{b^{b^{b^b}}}}$

\noindent 
Web: 
\href{https://www.researchgate.net/profile/Vahagn-Mikaelian}{researchgate.net/profile/Vahagn-Mikaelian}

 \end{document}